\documentclass[12pt]{amsart}

\usepackage[colorlinks,citecolor=blue,urlcolor=blue,linkcolor=blue]{hyperref}
\usepackage{amsmath}
\usepackage{amsfonts}
\usepackage{amssymb}
\usepackage{verbatim}
\usepackage[left=2.9cm,right=2.9cm,top=3cm,bottom=3cm]{geometry}
\usepackage{soul}
\usepackage{graphicx}

\newcommand{\R}{\mathbb R}

\newtheorem{rem}{Remark}

\newtheorem{prop}{Proposition}

\newtheorem{example}{Example}

\newtheorem{corollary}{Corollary}

\title[Reversibility and involutions]
{Reversible Markov kernels \\ and involutions on product spaces}

\author[M.\,Piccioni] {Mauro Piccioni}
\address{University of Rome La Sapienza, Department of Mathematics, Piazzale Aldo Moro, 5, 00185, Rome, Italy}
\email{mauro.piccioni@uniroma1.it}

\author[J.\,Weso{\l}owski]{Jacek Weso\l owski}
\address{Warsaw University of Technology, Faculty of Mathematics and Information Science, Koszykowa 75, 00-602, Warsaw, Poland}
\email{jacek.wesolowski@pw.edu.pl}

\begin{document}
	
	\begin{abstract}
		In this paper the relations between independence preserving (IP) involutions and reversible Markov kernels are investigated. We introduce an involutive augmentation $\mathbf H=(f,g_f)$ of a measurable function $f$ and relate the IP property of $\mathbf H$ to $f$-generated reversible Markov kernels. Various examples appeared in the literature are presented as particular cases of the construction. In particular, we prove that the IP property generated by the (reversible) Markov kernel of random walk with a reflecting barrier at the origin  characterizes  geometric-type laws.
	\end{abstract}
	
	\maketitle
	
	
	
	\section{Introduction}\label{intro}
	
	In a number of  instances,  a measurable function $\bf{H}$ is specified on a  measurable product space  $\mathcal X \times \mathcal U$,  with the following properties: 
	
	i) $\bf{H}$ is an involution on $\mathcal X \times \mathcal U$ onto itself, i.e. it is invertible with $\bf {H}^{-1}=\bf{H}$;
	
	ii) $\bf{H}$ preserves a product measure $\mu \otimes \nu$ on $\mathcal X \times \mathcal U$.
	
	More generally there has been a revival of interest in independence preserving (IP) maps, i.e. maps $F:\mathcal X \times \mathcal U \to\widetilde{\mathcal X} \times \widetilde{\mathcal U}$ for which there exists a quadruple of probability measures $\mu$, $\nu$, $\widetilde{\mu}$, $\widetilde \nu$ on respective spaces, such that $F(\mu\otimes\nu)=\widetilde{\mu}\otimes\widetilde{\nu}$. They play a key role in models of integrable probability, such as: random lattice equation models, see e.g. Croydon and Sasada \cite{CroSas20}, random polymer models, see e.g. Sepp\"al\"ainen \cite{Sep12}, Chaumont and Noack \cite{ChaNoa18} and references therein, and random walks on symmetric matrices, see e.g. Arista, Bisi and O'Connell \cite{AriBisOCo24}. In particular, Sasada and Uozumi \cite{SasUoz22} identified recently an interesting family of involutive IP maps, related to so-called $[2:2]$ quadrirational Yang-Baxter maps. 
	
	The goal of this paper is to relate involutive IP maps to reversible Markov transition kernels.
	
	Denote by $\mathcal F$ a $\sigma$-algebra of subsets of $\mathcal X$. A Markov transition kernel $K:\mathcal X\times \mathcal F$ is a function such that $K(\cdot,B)$ is $\mathcal F$-measurable for any $B\in \mathcal F$ and $K(x,\cdot)$ is a probability measure on $\mathcal X$ for all $x\in\mathcal X$. Next let $\mu$ be a probability measure of $\mathcal X$. Recall that a Markov transition kernel $K$ is said to be $\mu$-reversible iff 
	$$
	\mu(dx)\,K(x,dy)=K(y,dx)\,\mu(dy),\quad x,y\in\mathcal X^2.
	$$
	
	In the following we will be interested in transition kernels $K$ generated  by a family of random functions $\{f(x,U), x \in \mathcal X\}$, meaning that
	\begin{equation}\label{ker}
		K(x,\cdot)=\mathrm{Pr}(f(x,U) \in \cdot),\quad x \in \mathcal {X},
	\end{equation}
	where $U$ is a $\nu$ distributed random variable. Clearly, a transition kernel $K$, defined through \eqref{ker}, is $\mu$-reversible iff for $X\sim \mu$, independent of $U$, we have 
	\begin{equation}\label{xfu}(X,f(X,U))\stackrel{d}{=}(f(X,U),X).\end{equation}
	
	Notice that assumption i) alone trivially implies that any probability measure on the product space $\mathcal X \times \mathcal U$ (including any product measure $\mu \otimes \nu$) is preserved by $\mathbf H^2$, this being equal to the identity. But once $\bf{H}$ has the additional property ii), taking $X$ and $U$ independent random elements with laws $\mu$ and $\nu$ respectively and setting $(Y,V)={\bf{H}} (X,U)$, we have by i)
	$$
	(X,U;Y,V)=({\bf{H}} ^{-1} (Y,V);Y,V)=({\bf{H}} (Y,V);Y,V),
	$$
	and by ii) that $(X,U)$ and $(Y,V)$ are identical in law. This entails the following equality in law 
	$$
	(X,U;Y,V)\stackrel{d}{=}({\bf{H}} (X,U);X,U)=(Y,V;X,U),
	$$
	implying that $(X,U)$ is exchangeable with $(Y,V)$. Writing ${\bf{H}} =(f,g)$, the above relation implies
	\eqref{xfu}, i.e. $K$ is $\mu$-reversible with $K$ defined as in \eqref{ker}. 
	
	Here is the first example, that aroused our interest on the subject.
	
	\begin{example} The Matsumoto-Yor example. With $\mathcal {X}=\mathcal {U}=\mathbb {R}_+$ define 
		$$
		f(x,u)=\frac {1}{x+u},\,\,\,\, g(x,u)=\frac {1}{x}-\frac {1}{x+u}.
		$$
		It is immediately checked that $\mathbf H=(f,g)$ is an involution on $\mathbb R_+^2$. Matsumoto and Yor in \cite{MatYor01} proved that $\mathbf H$ preserves the product law $\phi_{\alpha, \lambda} \otimes \gamma_{\alpha, \lambda}$, where
		\begin{equation}\label{gig}
			\phi_{\alpha, \lambda}(dx)= C(\alpha, \lambda) x^{-\alpha-1}\exp \{-\lambda (x+x^{-1})\}dx,\,\,\, x>0,
		\end{equation}
		$C(\alpha, \lambda)$ being a suitable normalizing constant, and
		$\gamma_{\alpha, \lambda}$ is the gamma distribution with shape parameter $\alpha$ and rate parameter $\lambda$, with $\alpha, \lambda >0$. Letac and Weso\l owski \cite{LetWes00}, proved that this property characterizes the pair of distributions. According to the standard parametrization \eqref{gig} is a GIG distribution with parameters $(-\alpha, \sqrt{2\lambda}, \sqrt {2\lambda})$). As a consequence the transition kernel generated by $\{f(x,U), x>0\}$, with $U \sim \gamma_{\alpha, \lambda}$, defined by 
		$$
		K(x,A)=\mathrm{Pr}(f(x,U) \in A)=\int_{A^{-1}-x}\gamma_{\alpha, \lambda}(y)dy,\,\,\, x>0,
		$$
		for any Borel set $A\subset \mathbb R_+$ is $\phi_{\alpha, \lambda}$-reversible; here $A^{-1}=\{v\in\mathbb R_+:\;\tfrac{1}{v}\in A\}$. 
	\end{example}
	
	Summing up, we see that any involutive IP map $\mathbf H=(f,g)$ generates a reversible transition kernel of the form \eqref{ker}. The goal of this note is to present a general result going in the opposite direction, aiming to build  IP maps  from reversible transition kernels of the form \eqref{ker}. The result relies on a set-theoretical construction, the involutive augmentation of a function $f: \mathcal {X} \times \mathcal {U} \to \mathcal {X}$  defined on a product space, presented in Section 2, where it is illustrated by a few relevant examples. In Section $3$ the relation with the reversibility of transition kernels is explored. In Section $4$ the reflected random walk example is studied in detail:  here the reader can also find a characterization of the product measures preserved by the corresponding  involutive augmentation.  Section $5$ is finally devoted to discuss a somewhat canonical choice of the involution, which is shown to be a particular instance of the so-called Rosenblatt transformation.
	
	\section{The basic construction}
	
	The following purely set-theoretical result guarantees that under some rather simple conditions a function $f: \mathcal {X} \times \mathcal {U} \to \mathcal {X}$ has an \emph{involutive augmentation} ${\bf{H}} =(f,g)$ on the product space $\mathcal {X} \times \mathcal {U}$. Here is the  construction.
	
	Define first the $f$-accessible set and the $f$-uniquely accessible set as
	$$
	A_f=\{(x,y)\in \mathcal {X}^2: \,\,\exists u \in \mathcal {U}\,\,\ s.t. \,\, y=f(x,u)\},
	$$
	$$
	U_f=\{(x,y)\in A_f: \,\,\exists ! u \in \mathcal {U}\,\,\ s.t. \,\, y=f(x,u)\},
	$$
	respectively. Next, for $(x,y) \in U_f$ we define $\sigma (x,y)$ to be the unique solution in $u$ of the equation $y=f(x,u)$, that is for all $(x,y)\in U_f$ and $u\in\mathcal U$
	\begin{equation}\label{sigma}
		\sigma(x,y)=u\quad\mbox{iff}\quad y=f(x,u).
	\end{equation}	
	
	\begin{prop}\label{basic} Consider a function $f: \mathcal {X} \times \mathcal {U} \to \mathcal {X}$ with the properties:
		
		i) $A_f$ is symmetric, i.e. $(x,y) \in A_f$ if and only if $(y,x) \in A_f$;
		
		ii) $A_f \setminus U_f \subset D:=\{(x,x): x \in \mathcal {X}\}$, the diagonal in $\mathcal X^2$.
		
		\noindent
		In this case denote
		$$
		R_f=\{(x,u) \in \mathcal X \times \mathcal U: (x,f(x,u)) \in U_f\}
		$$
		and define $g_f: \mathcal {X} \times \mathcal {U} \to \mathcal {U}$  by
		\begin{equation}\label{gfu}
			g_f(x,u)=\left\{\begin{array}{ll}\sigma(f(x,u),x), & (x,u) \in R_f,\\
				u, & \mbox{otherwise}.\end{array}\right.   
		\end{equation}
		
		Then $(f,g_f)$ is an involution on $\mathcal X \times \mathcal U$. Conversely, if $(f,g)$ is an involution on $\mathcal X \times \mathcal U$, then $g\vert_{R_f}=g_{f}\vert_{R_f}$.
	\end{prop}
	
	\begin{proof} 
		Suppose first $(x,u) \in R_f$, hence $(x,f(x,u))\in U_f$. By i) and ii) $U_f$ is symmetric. Therefore, $(f(x,u),x)\in U_f$, so  
		\begin{equation}\label{ffu}
			f(f(x,u),g_f (x,u))=f(f(x,u),\sigma(f(x,u),x))=x.
		\end{equation}
		
		Moreover,
		$$
		g_f(f(x,u),g_f(x,u))=\sigma(f(f(x,u),g_f(x,u)),f(x,u))=\sigma(x,f(x,u))=u.
		$$
		
		For $(x,u)\notin R_f$ we have $(x,f(x,u))\in A_f\setminus U_f$, i.e. in view of ii), $f(x,u)=x$. Consequently,
		$$
		f(f(x,u),g_f(x,u))=f(f(x,u),u)=f(x,u)=x
		$$ 
		and
		$$g_f(f(x,u),g_f(x,u))=g_f(f(x,u),u)=g_f(x,u)=u.$$
		
		On the other hand, if $(f,g)$ is an involution then 
		$$
		f(f(x,u),g(x,u))=x
		$$
		and thus for $(x,u)\in R_f$, i.e. $(x,f(x,u))\in U_f$, in view of symmetry of $U_f$ we have $(f(x,u),x)\in U_f$. Thus, definition \eqref{sigma} implies $g(x,u)=\sigma(f(x,u),x)$. 
	\end{proof}
	
	\begin{example} Example 1, continued. Let $f(x,u)=\frac {1}{x+u}$, for $(x,u)\in \mathbb R_+^2$. For $(x,y)\in\R_+^2$ there exists $u\in\mathbb R_+$ such that $y=\tfrac{1}{x+u}$  iff $xy<1$. Moreover, such  a $u$ is unique. Therefore,
		$$
		A_f=U_f=\{(x,y) \in \mathbb {R}_+^{2}: xy<1\}.
		$$
		is symmetric. For any $(x,y) \in U_f$, we see that $\sigma(x,y)=\tfrac {1}{y}-x$. Moreover for any $(x,u)\in\R_+^2$ we have $(x,\tfrac{1}{x+u})\in U_f$, i.e. $R_f=\R_+^2$, and \eqref{gfu} yields
		$$
		g_f(x,u)=\sigma(f(x,u),x)=\sigma(\tfrac{1}{x+u},x)=\tfrac{1}{x}-\tfrac{1}{x+u},\quad \forall (x,u)\in\R_+^2.
		$$
		So the Matsumoto-Yor involution ${\bf{H}}=(\tfrac{1}{x+u},\tfrac{1}{x}-\tfrac{1}{x+u})$ is the unique involutive augmentation of $f(x,u)=\tfrac{1}{x+u}$. 
		
		On the other hand, defining  $\tilde f(x,u)=\tfrac{1}{u}-\tfrac{1}{x+u}=\tfrac{x}{u(x+u)}$, $(x,u)\in\mathbb R_+^2$, we see that for any $x,y>0$ there exists a unique $u>0$ such that $y=\tfrac{x}{u(x+u)}$, hence $A_{\tilde f}=U_{\tilde f}=R_{\tilde f}=\R_+^2$. Indeed, for any $(x,y)\in\mathbb R_+^2$ we have 
		$$\sigma(x,y)=\tfrac{\sqrt{xy(4+xy)}-xy}{2y}.$$
		As a consequence \eqref{sigma} yields
		\begin{align*}
			g_{\tilde f}(x,u)&=\sigma(\tilde f(x,u),x)=\sigma(\tfrac{x}{u(x+u)},x)=\tfrac{\sqrt{\tfrac{x^2}{u(x+u)}\left(4+\tfrac{x^2}{u(x+u)}\right)}-\tfrac{x^2}{u(x+u)}}{2x}\\
			&=\tfrac{\sqrt{4u(x+u)+x^2}-x}{2u(x+u)}=\tfrac{1}{x+u}.
		\end{align*}
		So the swapped Matsumoto-Yor involution ${\tilde {\bf{H}}}=(\tfrac{1}{u}-\tfrac{1}{x+u},\,\tfrac{1}{x+u})$ is the unique involutive augmentation of $\tilde f(x,u)=\tfrac{1}{u}-\tfrac{1}{x+u}$. 
	\end{example}
	
	\begin{example}
		Let $\mathcal X=\mathcal U=\Omega_+$, the cone of symmetric real positive definite matrices of fixed dimension. Let $f:\Omega_+^2\to\Omega_+$ be defined by
		$$
		f(x,u)=\mathbb P_{(\mathrm I+x)^{-1/2}}(u),\quad x,u\in\Omega_+,
		$$
		where $x^{1/2}$ is the unique element of $\Omega_+$, such that $x^{1/2}x^{1/2}=x$,  $\mathbb P_x(y)=xyx$ is the so-called quadratic representation of the respective Jordan algebra and $\mathrm I$ is the identity matrix. Then $A_f=U_f=\Omega_+^2$ and 
		$$
		\sigma(x,y)=\mathbb P_{(\mathrm I+x)^{1/2}}(y), \quad (x,y)\in \Omega_+^2.
		$$
		Following \eqref{gfu} we obtain
		$$
		g_f(x,u)=\sigma(f(x,u),x)= \mathbb P_{(\mathrm I+f(x,u))^{1/2}}(x)=\mathbb P_{(\mathrm I+\mathbb P_{(\mathrm I+x)^{-1/2}}(u))^{1/2}}(x),\quad (x,u)\in\Omega_+^2.
		$$
		Consequently, $\mathbf H:\Omega_+^2\to\Omega_+^2$
		$$
		{\mathbf H}(x,u)=\left(\mathbb P_{(\mathrm I+x)^{-1/2}}(u),\;\mathbb P_{(\mathrm I+\mathbb P_{(\mathrm I+x)^{-1/2}}(u))^{1/2}}(x)\right)
		$$
		is the involutive augmentation of $f$. 
		
		Note that $\mathbf H$ is an IP involution  for Wishart and Kummer random matrices, and a related characterization is also available (see \cite{KolPil20}).
	\end{example}
	
	\begin{rem}\label{KdV} The symmetry of $A_f$ is clearly a necessary condition for the existence of an involutive augmentation $\bf {H}$ $=(f,g)$ of a function $f: \mathcal X \times \mathcal U \to \mathcal X$, but  the remaining part of the assumption in Proposition \ref{basic} is far from being necessary. 
		
		Consider as an example the mapping 
		$$
		f(x,u)=u-(x+u)^+=u\wedge(-x),
		$$ 
		with $\mathcal X=\mathcal U =\mathbb Z$. Observe that $A_f=\{(x,y):x+y \leq 0 \}$ is symmetric, hence assumption i) of Proposition 1 is satisfied. However, this is not true for assumption ii). Indeed, whereas for a given $y<-x$ the equation $f(x,u)=y$ has the unique solution $u=y$, for $y=-x$ any $u \geq -x$ is a solution the equation $f(x,u)=y$. In order to make the function ${\bf{H}}=(f,g)$ an involution on $\mathbb Z^2$ one needs
		\begin{equation}\label{possible}
			g(x,u)=x,\,\,\, x+u<0,\,\,\, g(x,u)=x+s(x,u)\geq x,\,\,\, x+u\geq 0,
		\end{equation}
		$s(x,u)$ being a non-negative function defined on the pairs $(x,u)$ such that $x+u \geq 0$, and on this domain it has to satisfy
		\begin{equation}\label{necandsuff}
			s(-x,x+s(x,u))=x+u=(x+u)^+    
		\end{equation}
		The choice $s(x,u)=(x+u)^+$, entails
		$$
		g(x,u)=x+(x+u)^+=:g_1(x,u),
		$$
		which corresponds to the so-called \emph{ultra-discrete KdV (Korteweg de Vries) equation} with parameters $J=0$ and $K=+\infty$, discussed in Croydon and Sasada, \cite{CroSas20}, page 14, where the product measures preserved by this map are also described. 
		Note that if $h$ is an involution on the set of non-negative integers  $\mathbb N_0$, then $s(x,u)=h((x+u)^+)$ is a legitimate choice for $s$.  For example one can take $h$ defined by 
		$$
		h(n)=n-(-1)^n+\mathbf 1_{\{0\}}(n),\quad n\in \mathbb N_0,
		$$
		with the corresponding involutive augmentation $(f,g)$
		$$
		f(x,u)=u \wedge (-x),\,\, g_2(x,u)=x+(x+u)^+-(-1)^{(x+u)^+}+\mathbf 1_{\{0\}}((x+u)^+),\quad x,u\in \mathbb Z.
		$$
	\end{rem}

	\section{Reversibility and involutions}
	
	In this section probabilities enter the play.
	
	\begin{prop}\label{twobasic} Suppose that a family of random functions $\{f(x,U), x \in \mathcal X\}$,  where $U$ has distribution $\nu$,  generates as in \eqref{ker} a $\mu$-reversible kernel $K$, where $\mu$ is a probability measure on $\mathcal X$. In addition, suppose that $f$ satisfies the assumptions of Proposition \ref{basic} and that the function $g_f$ constructed therein is measurable. Then the involution $\mathbf H$ $=(f,g_f)$ preserves the product law $\mu \otimes \nu$.
	\end{prop}
	
	\begin{proof}  Let $(X,U)$ be distributed according to $\mu \otimes \nu$, and define $(Y,V)={\bf H}(X,U)$, where $\bf {H}$=$(f, g_f)$. Recall that when $(x,y) \in A_f \setminus D$ then $y=f(x,u)$ with $u=\sigma(x,y)$ and $x=f(y,v)$ with $v=\sigma(y,x)$. Finally call $F=\{(x,u) \in \mathcal X \times \mathcal U: f(x,u)=x\}$.
		
		For an arbitrary measurable set $B\subset\mathcal X \times \mathcal U$ we have
		$$
		\mathrm{Pr}((X,U)\in B)=\mathrm{Pr}((X,U)\in B\cap F^c)+\mathrm{Pr}((X,U)\in B\cap F).
		$$
		Note that for $(x,u)\in F^c$ we have $(x,y) \notin D$, and thus
		\begin{align*}
			\mathrm{Pr}((X,U)\in B\cap F^c)=&\mathrm{Pr}((X,\sigma(X,Y))\in B\cap F^c)=\mathrm{Pr}((X,\sigma(X,f(X,U)))\in B\cap F^c)\\
			%\stackrel{\eqref{xfu}}
			{=}&\mathrm{Pr}((f(X,U),\sigma(f(X,U),X))\in B\cap F^c)\\
			=&\mathrm{Pr}((Y,\sigma(Y,X))\in B\cap F^c)=\mathrm{Pr}((Y,V)\in B\cap F^c),
		\end{align*}
		where the third equality above is a consequence of the reversibility assumption and the independence of $X$ and $U$, see \eqref{xfu}.
		
		Since $(f,g_f)$ is the identity map on $F$ we get
		$$
		\mathrm{Pr}((X,U)\in B\cap F)=\mathrm{Pr}((f(X,U), g_f (X,U))\in B\cap F)=\mathrm{Pr}((Y,V)\in B\cap F).
		$$
		Consequently, $(X,U)\stackrel{d}{=}(Y,V)$ and the result follows. 
	\end{proof}
	
	\begin{rem} Let $K$ be a transition kernel on $\mathcal X$ generated by a family of random functions $\{f(x,U),\, x \in \mathcal X\}$  (see \eqref{ker}),  where the random variable $U \sim \nu$ takes values in $\mathcal U$. Let $\mu$ be a $K$-stationary distribution, i.e. $\mu(dy)=\int\,K(x,dy)\,\mu(dx)$. Let the i.i.d. sequence $(U^t,\, t \in \mathbb {N}_0)$ be drawn from $\nu$, and let $X^0$ be an independent $\mu$-distributed random variable. The recursion 
		\begin{equation}\label{prima}
			X^{t+1}=f(X^{t},U^{t}),\,\, t \in \mathbb {N}_0,
		\end{equation}
		defines a stationary (i.e. translation invariant) Markov process $(X^t, t \in \mathbb N_0)$ with transition kernel $K$. If the assumptions of Proposition \ref{twobasic} hold, the function $g_f$ is well defined, and the sequence 
		\begin{equation}\label{seconda}
			V^{t}=g_f(X^{t}, U^{t}),\,\, t \in \mathbb {N}_0,
		\end{equation}
		has the property that $V^{t} \sim \nu$ and it is independent of $X^{t+1}$, for $t \in \mathbb {N}_0$. In particular $V^0=g_f(X^0, U^0)$, $X^1$ and $(U^t, t \in \mathbb {N})$, are independent. Note that this set of random variables  is used for  constructing  random variables  $X^{t}$ and $V^t$ according to \eqref{prima} and \eqref{seconda}, for all $t \in \mathbb {N}$. With a straightforward recursive argument we conclude that $V^t$ are i.i.d. $\nu$-distributed for $t \in \mathbb N_0$, and thus $(U^t,\, t \in \mathbb {N}_0)\stackrel{d}{=}(V^t,\, t \in \mathbb {N}_0)$. Consequently, the latter can replace the former  in \eqref{prima}, in order to construct a distributional copy of $(X^t, t \in \mathbb N_0)$. 
		
		This construction can be better visualized by defining a $\mathcal X \times \mathcal U$-valued random field $((X_n^t, U_n^t), (n,t) \in \mathbb {N}^2)$ evolving according to
		\begin{equation}\label{crosas}
			(X^{t+1}_n,U^{t}_{n})= \mathbf H (X^{t}_n, U^{t}_{n-1}), \, (n,t) \in \mathbb N \times \mathbb N_0,
		\end{equation}
		where $\mathbf H=(f,g_f)$ and with the initial values $(X_n^0, n \in \mathbb N)$ i.i.d. drawn from $\mu$, and  $(U_{0}^{t},\, t \in \mathbb {N}_0)$ i.i.d. drawn from $\nu$, independently. Indeed, renaming  $(U^t,\,X^t,\,V^t)$ as $(U^{t}_0,\,X^t_1,\, U^{t}_{1})$ for $t \in \mathbb N_0$, the relation \eqref{crosas} for $n=1$ is rewritten as
		\begin{equation}\label{rosas}
			(X^{t+1},V^t)=\mathbf H(X^{t},U^t), t \in \mathbb N_0,
		\end{equation}
		which is nothing but \eqref{prima} and \eqref{seconda} written together. From the above discussion it follows that for any $n \in \mathbb N$ the column $(X_n^t,\, t \in \mathbb N_0)$ is a stationary Markov chain with transition kernel $K$  generated by $f$  as in  \eqref{ker}, marginal distribution $\mu$ and the column $(U^{t}_{n},\, t \in \mathbb N_0)$ has i.i.d. $\nu$-distributed components. 
		
		Incidentally, since the construction can be presented in an entirely equivalent way by exchanging rows with columns, it is easily established that for any $t \in \mathbb N_0$ the row $(X_n^t,\, n \in \mathbb N)$ has i.i.d. $\mu$-distributed components, and the row $(U^t_n,\, n \in \mathbb N_0)$ is a stationary Markov chain with marginal distribution $\nu$ and with the transition kernel on $\mathcal {U}$ generated  by the family of random functions $(g_f(X,u), u \in \mathcal U)$, with $X \sim \mu$. Moreover, by the translation invariance of the construction it is possible to construct a translation invariant random field on the whole lattice $\mathbb {Z}^2$ by means of Kolmogorov extension theorem, supported by configurations satisfying \eqref{crosas} for $(n,t) \in \mathbb {Z}^2$.
		
		This is essentially the content of part a) of Proposition 2.9 in Croydon and Sasada \cite{CroSas20}, who consider the random field $((X_n^t, U_n^t), (n,t) \in \mathbb {Z}^2)$ defined through \eqref{crosas}. This paper addresses mainly the problem of uniqueness of solutions of \eqref{crosas} for given initial random sequence $(X^0_n, n \in \mathbb Z)$. It appears that the above described invariance/independence properties, called Burke's property (see Section 2.2 therein), are essential in this context. 
	\end{rem}
	
	\begin{example}\label{ex:char} By the general arguments in Section \ref{intro}, taking into account the Matsumoto-Yor property, the kernel associated to the family of random functions $\{\frac {1}{x+U}, x>0\}$, with $U  \sim \gamma_{\alpha, \lambda}$ is $\phi_{\alpha, \lambda}$-reversible. On the other hand, by the characterization in \cite{LetWes00} and Proposition \ref{twobasic}, when $U$ is not gamma distributed, $K$ is never reversible: there is no law $\mu$ that makes it such. 
	\end{example}
	
	Characterization results of the kind described in Example \ref{ex:char} (see e.g. \cite{LW2024}, \cite{KW2024} for more recent ones) allow to determine the reversibility of certain transition kernels of the form \eqref{ker} without explicit calculations. 
	
	\begin{corollary}\label{fgf} Suppose that $f: \mathcal X \times \mathcal U \to \mathcal X$ is measurable and satisfies the assumption of Proposition \ref{basic}, with the function $g_f$ measurable as well. Let $\nu$ be a distribution on $\mathcal U$. There exists a distribution $\mu$ on $\mathcal X$ such that $\mu\otimes\nu$ is invariant under $\mathbf H=(f, g_f)$ if and only if  the transition kernel
		$$
		K(x,\cdot)=\mathrm{Pr}(f(x,U) \in \cdot)
		$$
		with $U \sim \nu$, is $\mu$-reversible.
	\end{corollary}
	\begin{proof}
		Sufficiency is given in Proposition \ref{twobasic} and necessity follows since, as it has already been observed in the introduction,  if $\mathbf H$ is an involution preserving $\mu\otimes\nu$ then \eqref{xfu} holds true. 
	\end{proof}
	
	\begin{example} With $\mathcal X=\mathcal U=(0,1)$ define
		$$
		f(x,u)=\tfrac {1-u}{1-ux}
		$$
		fulfilling the assumption of Proposition \ref{basic}. Then $g_f(x,u)=1-ux$. In \cite{SesWes03} it is proved that the product measures that are preserved by $\mathbf H=(f,g_f)$ form the set 
		$$
		\mathcal I= \{\mathrm{Beta}_I(p,q) \otimes \mathrm{Beta}_I(p+q,p),\;p, q >0\}.
		$$
		(Recall that $\mathrm{Beta}_I(a,b)$ law is defined by the density $\propto x^{a-1}(1-x)^{b-1}\,\mathbf 1_{(0,1)}(x)$.)  
		As a consequence of Remark \ref{fgf}, the kernel $\mathrm{Pr}\left(\tfrac {1-U}{1-xU} \in \cdot\right),\;x \in (0,1)$, with $U \sim \mathrm{Beta}_I(a,b)$, is reversible if and only if $a>b$. 
	\end{example}
	
	\begin{example} The Beta walk. Diaconis and Freedman, \cite{DiaFre99}, studied the following Markov chain. Let $\mathcal X=(0,1)$ and $\mathcal U=\{0,1\}\times (0,1)$. Define $f: \mathcal X \times \mathcal U \to \mathcal X$ as follows
		\begin{equation}\label{dirf}
			f(x,(u_0,u_1))=(1-u_1)x+u_0 u_1.
		\end{equation}
		Sethuraman \cite{Set94}, proved that for independent $X \sim \mathrm{Beta}_I(a_0, a_1)$, $U_0 \sim \mathrm{Bernoulli}(\frac {a_1}{a_0+a_1})$ and $U_1 \sim \mathrm{Beta}_I(1,a_0+a_1)$ the following identity in law holds
		$$
		f(X,(U_0,U_1)) \stackrel{d}{=} X.
		$$
		Generalizations  appeared in \cite{HitLet14}. It is readily verified  that the assumptions of Proposition \ref{basic} hold true for  $f$ defined in \eqref{dirf}. Moreover $A_f=U_f=\{(x,y) \in \mathcal X^2: x \neq y\}$ and
		$$
		\sigma(x,y)=(0,1-\tfrac{y}{x})\mathbf 1_{0<y<x<1}+(1,\tfrac{y-x}{1-x})\mathbf 1_{0<x < y<1},\,\,\, (x,y) \in U_f.
		$$
		Therefore, the function $g_f: \mathcal X \times \mathcal U \to \mathcal U$, defined as 
		$$
		g_f(x,(u_0,u_1))=\sigma(f(x,(u_0,u_1)),x),
		$$
		assumes the form
		\begin{equation}\label{added}
			g_f(x,(u_0,u_1))=\left(1, \frac {u_1x}{1-x+u_1x}\right)(1-u_0)+\left (0,\frac {u_1(1-x)}{u_1(1-x)+x}\right )u_0
		\end{equation}
		and produces the involutive augmentation $\mathbf H=(f,g_f)$ of $f$. 
		
		It is natural to ask if there are product laws which are preserved by $\mathbf H$, i.e. marginal laws for $X$ and $U=(U_0,U_1)$, mutually independent, such that
		\begin{equation}\label{yvv}
			(Y;V_0,V_1):=(f(X,(U_0,U_1)); g_f(X,(U_0,U_1)))\stackrel{d}{=}(X;U_0,U_1). 
		\end{equation}
		It is easily realized that the answer is negative. Indeed, from \eqref{dirf}, \eqref{added} and \eqref{yvv} one gets $1-U_0=V_0\stackrel{d}{=}U_0$ and $YV_0=X(1-U_1)(1-U_0)\stackrel{d}{=}XU_0$. Consequently,
		$$
		\mathbb E\,X\,\mathbb E\,(1-U_0)(1-U_1)=\mathbb E\,X\,\mathbb E\,U_0=\mathbb E\,X\,\mathbb E\,(1-U_0),
		$$
		which implies that either $\mathrm{Pr}(X=0)=1$  or $\mathrm{Pr}(U_1=0)=1$.

		On the other hand it is  directly verified that, under the distributions for $X, U_0$ and $U_1$ assumed by Sethuraman, exchangeability of the joint law of $X$ and $f(X,(U_0,U_1))$ never occurs. This is in agreement with Proposition \ref{twobasic} and the lack of product laws preserved by $\bf {H}$.
	\end{example}
	
	\begin{rem}  Consider $f, g_1, g_2$ as in Remark \ref{KdV}. For $\ell \in 2\mathbb N$ and $0<\theta <1$ define the probability measures $\mu_{\theta, \ell}$ and $\nu_{\theta, \ell}$ defined by 
		\begin{align*}
			\mu_{\theta, \ell} (\{x\})&=\frac{1}{Z(\theta, -\ell, \ell)} \theta^x,\qquad  x\in\{-\ell,-\ell+1,\ldots,\ell\}\\
			\nu_{\theta, \ell}(\{u\})&=\frac{1}{Z(\theta, -\ell, +\infty)} \theta^u,\qquad u\in\{-\ell,-\ell+1,\ldots\}
		\end{align*}
		with 
		$$
		Z(\theta, m, n)=\sum_{m}^{n} \theta^{i},\,\,\, m \in -\mathbb N,\; n \in \mathbb {N} \cup \{+\infty\}.
		$$
		Recall that $(f,g_1)$ and $(f,g_2)$ are involutive augmentations of $f$, but the assumptions of Proposition \ref{basic} does not hold for $f$. The reader can easily verify that the product law $\mu_{\theta, \ell} \otimes \nu_{\theta, \ell}$ is preserved by the involution $(f, g_1)$, but not by the involution $(f,g_2)$. 
		Actually the former fact is mentioned in Croydon and Sasada \cite{CroSas20}, Proposition 3.3, and implies that the kernel $K$, generated by random functions
		$\{f(x,U), x \in \mathbb Z\}$ with $\nu_{\theta, \ell}$-distributed $U$, is $\mu_{\theta, \ell}$-reversible. As a consequence the conclusion of Proposition \ref{twobasic} is still true for the 
		involution $(f, g_1)$ but not for the involution $(f,g_2)$. For a general $\mu$-reversible kernel $K$, generated by  random functions $\{f(x,U), x \in \mathcal X \}$ with $U \sim \nu$, to which Proposition \ref{basic} does not apply, we are not aware of any immediate criterion for the construction of an involutive augmentation of $f$ preserving the product measure $\mu \otimes \nu$.
	\end{rem}
	
	\section{A characterization of independence for a reflecting random walk.}
	
	The next example, which is interesting by itself, explains the reason for allowing multiple solutions of the equation  $f(x,u)=x$ for any fixed $x \in \mathcal {X}$, according to assumption ii) in Proposition \ref{basic}. The example is related to Burke's theorem in the discrete setting, see Keane and O'Connell \cite{KeaOco08}.
	
	\begin{example} Reflecting random walk with a reflecting barrier at the origin. The following transition kernel $\{K(x, \cdot), x \in \mathcal X\}$ is specified on the set $\mathcal X=\mathbb {N}_0$ of non-negative integers
		$$
		K(x,x+1)=p,\quad x \in \mathbb {N}_{0},\quad K(x,x-1)=q,\quad K(x,x)=r,\quad x \in \mathbb {N},\quad K(0,0)=q+r,
		$$
		with $0<p<q$ and $r=1-p-q \geq 0$. It is easily checked that this kernel is $\mu$-reversible with $\mu=\mathrm{geo}(\theta)$ defined by $\mu(x)=(1-\theta) \theta^{x},\;x\in\mathbb {N}_0$, where $\theta=p/q$. Clearly, $K(x,\cdot)=\mathrm{Pr}(f(x,U)\in\cdot)$, where $f(x,u)=(x+U)^+$, $x\in\mathbb N_0$, with $\mathcal U=\{-1,0,1\}$ and $\mathrm{Pr}(U=1)=p$, $\mathrm{Pr}(U=-1)=q$ and $\mathrm{Pr}(U=0)=r$. In view of Proposition \ref{twobasic} we conclude that $(X+U)^+$ has the same geometric law as $X$. The $f$-accessible set is given by $$A_f=\{(x,x\pm 1),\, x\in\mathbb N \;,(0,1),\;(x,x),\;x\in\mathbb N_0\},$$
		unless $r=0$, in which case all the diagonal pairs $(x,x)$, with $x>0$ are removed from $A_f$.
		Moreover, in both cases $U_f=A_f\setminus\{(0,0)\}$
		%=(x,u)\in\mathbb N\times\mathcal U\cup\{(0,1)\},$$
		and $\sigma (x,y)=y-x$ for $(x,y)\in U_f$.  Indeed, the only lack of uniqueness occurs when $x=y=0$, since $f(0,-1)=f(0,0)=0$. Anyway Proposition \ref{basic} can be applied and yields
		\begin{equation}\label{posit}
			g_f(x,u)=\left\{\begin{array}{ll}\sigma ((x+u)^+,x)=-u, & (x,(x+u)^+)\in U_f\color{black}\\
				u, & (x,u)\in\{(0,0),\,(0,-1)\}.\end{array}\right. 
		\end{equation}
		Both for $u=-1$ and $u=0$ the equation
		$$
		f(f(0,u),v)=f(0,v)=v^+=0.
		$$
		has two solutions $v=0, -1$: this means that in order to define $g$ with the property that $(f,g)$ is an involution, there are two options to define $g(0,0)$ and $g(0,-1)$, either with the choice in \eqref{posit}, that gives 
		\begin{equation}\label{firstchoice}
			g_f(0,0)=0,\,\,\, g_f(0,-1)=-1;
		\end{equation}
		or with swapped values. Since $(x+u)^-=0$, except for $x=0, u=-1$, when it is equal to $1$,
		we can write
		$$
		g_f(x,u)=-u-2(x+u)^-, (x,u)\in\mathcal X\times\mathcal U.
		$$
		Applying Proposition \ref{twobasic} we obtain that \begin{itemize}
			\item[i)] $g_f (X,U)=-U-2(X+U)^-$ is independent of $f(X,U)=(X+U)^+$; 
			\item[ii)] $g_f (X,U)$  has the same law as $U$.
		\end{itemize}    
		
		Property {\em i)} characterizes the laws of $X$, as it results from the next proposition, that also singles out the cases in which {\em i)} holds but not {\em ii)}. 
	\end{example}
	
	\begin{prop}
		Let $X$ be an $\mathbb N_0$-valued random variable such that $\mathrm{Pr}(X=0)<1$ and let $U$ be an independent $\{-1,0,1\}$-valued random variable with $\mathrm{Pr}(U=1)=p$, $\mathrm{Pr}(U=0)=r$ and $\mathrm{Pr}(U=-1)=q$ and $pq>0$. If the random variables
		\begin{equation}\label{invo}
			Y=(X+U)^+\quad\mbox{and}\quad V=-U-2(X+U)^-.
		\end{equation}
		are independent then necessarily $p<q$ and
		\begin{enumerate}
			\item[i)] either $r>0$ and $X\sim\mathrm{geo}\left(\tfrac{p}{q}\right)$, in which case 
			\begin{equation}\label{identity}
				(X,U)\stackrel{d}{=}(Y,V);
			\end{equation}
			
			\item[ii)] or $r=0$ and the conditional distributions of $X$ given $X\in 2\mathbb N$ and of $X$ given $X\in 2\mathbb N+1$ are geometric with the failure rate $\rho^2$, i.e. \begin{equation}\label{geom}
				\mathrm{Pr}(X=2k|X\in 2\mathbb N_0)=(1-\rho^2)\rho^{2k}=\mathrm{Pr}(X=2k+1|X\in 2\mathbb N_0+1),\quad k\in\mathbb N_0,
			\end{equation} where 
			$$
			\rho =
			%\frac {\mathrm{Pr}(U=1)P(V=1)}{P(U=-1)P(V=-1)}=
			\sqrt {\tfrac {p\,p'}{q\,q'}}\in(0,1)
			$$
			with $p':=\mathrm{Pr}(V=1)$ and $q':=\mathrm{Pr}(V=-1)$. Moreover,
			$$
			\mathrm{Pr}(X \in 2\mathbb N_0+1)=p',\quad \left(\mbox{i.e.}\quad \mathrm{Pr}(X\in 2\mathbb N_0)=q'\right),
			$$
			whereas
			\begin{equation}\label{solve}
				\mathrm{Pr}(Y\in 2\mathbb N_0)=q,\quad\left(\mbox{i.e.}\quad \mathrm{Pr}(Y\in 2\mathbb N_0+1)=p\right).
			\end{equation}
			In case $p'=p$ (i.e. $q'=q$), the  distribution of $X$ is geometric with failure rate $\rho=\frac {p}{q}$, and \eqref{identity} still holds.
		\end{enumerate}
	\end{prop}
	\begin{proof}
		
		By the definition of $(Y,V)$ in \eqref{invo}, and the fact that this pair of random variables is an involutive function of $(X,U)$ the following identities between events hold: 
		$$\{X=0, U=-1\}=\{Y=0, V=-1\}$$ 
		and for any $k \in \mathbb N_0$
		\begin{align*}
			&\{X=k, U=0\}=\{Y=k, V=0\}, \\
			&\{X=k+1, U=-1\}=\{Y=k, V=1\},\\
			&\{X=k,U=1\}=\{Y=k+1, V=-1\}, 
			&.
		\end{align*}
		Taking the assumed independence of $X$ and $U$ and of $Y$ and $V$ into account, we have thus
		\begin{equation}\label{fourth}
			\mathrm{Pr}(X=0)\,q=\mathrm{Pr}(Y=0)\,q'
		\end{equation} and for $k \in \mathbb{N}_0$ 
		\begin{align}\label{first}
			& \mathrm{Pr}(X=k)\,\mathrm{Pr}(U=0)=\mathrm{Pr}(Y=k)\,\mathrm{Pr}(V=0),\\
			\label{second}
			&\mathrm{Pr}(X=k+1)\,q=\mathrm{Pr}(Y=k)\,p',\\
			\label{third}
			&\mathrm{Pr}(X=k)\,p=\mathrm{Pr}(Y=k+1)\,q'.
		\end{align}
		
		Summing \eqref{first} over $k$ we have $\mathrm{Pr}(U=0)=\mathrm{Pr}(V=0)=r$. That is, $p+q=p'+q'$.
		Moreover, summing \eqref{second} over all $k$ we obtain 
		\begin{equation}\label{uv1}
			p'=\mathrm{Pr}(X \geq 1)\,q=(1-\mathrm{Pr}(X=0))\,q.
		\end{equation}
		By assumption $\mathrm{Pr}(X=0)<1$ from \eqref{uv1} one gets $0<p'<q$ and thus $0<p<q'$.
		
		In view of \eqref{first}, in case $r>0$, the random variables $X$ and $Y$ have the same law. It is easily deduced from \eqref{second} and \eqref{third} that the support of their common law is the whole $\mathbb {N}_0$. Then, multiply \eqref{second} and \eqref{third} side-wise for $k=0$ and use the property $\mathrm{Pr}(X=0)\mathrm{Pr}(X=1)>0$, to conclude that
		\begin{equation*}
			p\,q=p'\,q'.
		\end{equation*}
		Since by \eqref{fourth} we have $q=q'$, it follows that $U\stackrel{d}{=}V$. Consequently, $p<q$ and from either \eqref{second} or \eqref{third} we have that $X$ and $Y$ are geometric with failure probability $\frac {p}{q}$, as stated in i).
		
		It remains to check what happens for $r=0$. The relations \eqref{second} and \eqref{third} give in this case, for any $k \in \mathbb N_0$
		\begin{equation*}\label{reqzero}
			\mathrm{Pr}(X=k+2)=\rho^2\, \mathrm{Pr}(X=k),\qquad \mathrm{Pr}(Y=k+2)=\rho^2\, \mathrm{Pr}(Y=k)
		\end{equation*}
		where 
		$$
		\rho^2= \tfrac {p\,p'}{q\,q'}\in(0, 1),
		$$
		%\frac {1-p_0}{1+\frac {q}{p}p_0} 
		where we took into $0<p<q'$ and $0<p'<q$. 
		%Since by \eqref{uv1} and \eqref{uv2}
		%$$
		%P(V=1)=(1-p_0)q,\,\,\, P(V=-1)=p_0q+p.
		%$$
		Consequently, \eqref{geom} holds true. Moreover,
		summing \eqref{second} over $k\in 2 \mathbb N_0$ we obtain
		\begin{equation}\label{xoddyeven}
			\mathrm{Pr}(X \in 2\mathbb N_0+1)\,q=\mathrm{Pr}(Y \in 2\mathbb N_0)\,p',
		\end{equation}
		whereas summing \eqref{third} over $k\in 2\mathbb N_0+1$ we obtain
		\begin{equation}\label{xevenyodd}
			\mathrm{Pr}(X\in 2\mathbb N_0)\,p=\mathrm{Pr}(Y\in 2 \mathbb N_0+1)\,q'.
		\end{equation}
		Moreover,\begin{equation}\label{dai}
			\mathrm{Pr}(X \in 2\mathbb N_0+1)+q=\mathrm{Pr}(Y\in 2 \mathbb N_0)+p'.
		\end{equation}
		Indeed
		\begin{align*}
			&\mathrm{Pr}(X \in 2\mathbb N_0+1)+q\\&=\mathrm{Pr}(X \in 2\mathbb N_0+1)+\left(\mathrm{Pr}(X\in 2\mathbb N_0)+\mathrm{Pr}(X\in 2\mathbb N_0+1)\right)\,q\\
			&\stackrel{\eqref{xoddyeven}}{=}\mathrm{Pr}(X \in 2\mathbb N_0+1)+\mathrm{Pr}(X\in 2\mathbb N_0)\,q+\mathrm{Pr}(Y \in 2\mathbb N_0)\,p'\\
			&=1-\mathrm{Pr}(X\in 2\mathbb N_0)\,p+\mathrm{Pr}(Y \in 2\mathbb N_0)\,p'\\
			&\stackrel{\eqref{xevenyodd}}{=}1-\mathrm{Pr}(Y\in 2 \mathbb N_0+1)\,q'+\mathrm{Pr}(Y \in 2\mathbb N_0)\,p'\\
			&=1-\mathrm{Pr}(Y\in 2 \mathbb N_0+1)+\left(\mathrm{Pr}(Y\in 2 \mathbb N_0+1)+\mathrm{Pr}(Y \in 2\mathbb N_0)\right)\,p'\\
			&=\mathrm{Pr}(Y\in 2 \mathbb N_0)+p'.
		\end{align*}
		
		Since $0<p'<q$, the identities \eqref {dai} and \eqref{xoddyeven} yield
		$$
		\mathrm{Pr}(Y \in2\mathbb N_0)=q,\quad\mbox{and}\quad  \mathrm{Pr}(X\in 2\mathbb N_0+1)=p',
		$$
		as stated.
		
		In case $p'=p$ (i.e. $q=q'$) we see that 
		$$
		(1-\rho^2)\,p'=(1-\rho)(1+\tfrac{p}{q})p=(1-\rho)\rho
		$$
		and 
		$$(1-\rho^2)\,q'=(1-\rho)(1+\tfrac{p}{q})q=1-\rho.
		$$
		Consequently, for all $k\in \mathbb N_0$ we have
		$$
		\mathrm{Pr}(X=2k)=\mathrm{Pr}(X=2k|X\in 2\mathbb N_0)\,q'=(1-\rho^2)\rho^{2k}q=(1-\rho)\rho^{2k}
		$$
		and
		$$
		\mathrm{Pr}(X=2k+1)=\mathrm{Pr}(X=2k+1|X\in 2\mathbb N_0+1)\,p'=(1-\rho^2)\rho^{2k}p=(1-\rho)\rho^{2k+1},
		$$
		i.e. $X$ is geometric with the failure rate $\rho$.
	\end{proof}
	
	\section{The Skorokhod and the Rosenblatt representation}
	
	There is a universal way to  generate  a transition kernel $K(x,dy)$, defined on an open interval $\mathcal X=(a,b)\subset\mathbb {R}$ (possibly infinite at one or both sides), by a Borel measurable function $f$, defined on $(a,b) \times (0,1)$ with values in $(a,b)$ and with $U$ uniform on $\mathcal U=(0,1)$: the Skorohod transformation (see e.g. Kifer \cite{Kif86}).
	
	Let us suppose that the distribution function $F_x(y)=K(x,(-\infty,y])$ associated to $K(x,\cdot)$ is such that $F_x(a)=0$, $F_x(b-)=1$, and $F_x$ is continuously increasing on $(a,b)$. The Skorokhod transformation represents the kernel $K(x, \cdot)$ with the family of random functions
	$$
	f(x,U)=F^{-1}_x(U), \quad x \in (a,b),
	$$ 
	where $U \mbox{ is uniform on }  (0,1)$.
	
	Next, for any choice of $x, y \in (a,b)$, the equation $F^{-1}_x(u)=y$  has the unique solution  $u=\sigma(x,y)=F_x(y)$: hence Proposition \ref{basic} can be applied, with $A_f=U_f=(a,b)$: the choice
	$$
	g_f(x,u)=\sigma (f(x,u),x)=F_{F^{-1}_x(u)}(x)
	$$
	is necessary and sufficient to construct an involution $(f,g_f)$ on $(a,b) \times (0,1)$. 
	
	Assume that the function $(a,b) ^2\ni (x,y)  \mapsto F_y(x) \in (0,1)$ is Borel measurable. Then, since for any $t \in (a,b)$
	$$
	\{(x,u): F_x^{-1}(u) \leq t \}= \{(x,u): u \leq F_x(t)\},
	$$
	and $F_x(t)$ is Borel measurable in $x \in (a,b)$ for any fixed $t \in (a,b)$, then $f$ is Borel measurable. Moreover, as a composition of the Borel measurable function $(x,y) \mapsto F_y(x)$ with the measurable transformation from $(a,b) \times (0,1)$ into $(a,b)^2$
	$$
	(x,u) \mapsto (x, F_x^{-1}(u)),
	$$
	also $g_f(x,u)$ turns out to be Borel measurable.
	
	Without the reversibility assumption the product probability $\mu \otimes \lambda$, $\lambda$ being the Lebesgue measure in $(0,1)$, has no reason to be preserved by $(f,g_f)$. But in case $\{K(x,\cdot), x \in (a,b)\}$ is $\mu$-reversible, Proposition  \ref{twobasic} implies that if $X$ and $U$ are independent and $\mu$ and $\lambda$-distributed, respectively, then $Y:=F^{-1}_X(U)$ is $\mu$-distributed (this is nothing but the Skorokhod result) and moreover it is independent of $V:=g_f(X,U)=F_Y(X)$, that has the distribution $\lambda$.
	
	Finally notice that, by the reversibility assumption, for any $x, y \in (a,b)$,
	$$
	G_y(x):=\mathrm{Pr}(X \leq x | Y=y)=\mathrm{Pr}(Y \leq x | X=y)=F_y(x)
	$$
	which implies
	$$
	g_f (x,u)=G_{F^{-1}_x(u)}(x).
	$$ 
	Thus the previous conclusion can be rephrased by saying that the random variables
	$$
	Y=F^{-1}_X(U), V=G_Y(X)
	$$
	are independent with distributions $\mu$ and $\lambda$, respectively. This is a particular instance of the so-called Rosenblatt transformation $R(x,y)$ (introduced in \cite{Ros52}), which for any pair of random variables $(X,Y)$ is defined by $R(x,y)=\mathrm{Pr}(Y\le y|X=x)$, and under suitable regularity conditions on their joint distribution, has the property that $X$ and $R(X,Y)\sim \lambda$ are independent random variables.
	The Rosenblatt transformation actually is defined in a more general multivariate setting and without any reversibility assumption.
	
	\begin{example} Consider the Gaussian kernel on $\mathbb {R}$, given by
		$$
		K(x,dy)=\mathrm N(\beta x, \sigma^2) (dy),
		$$
		where $|\beta|<1$ and $\sigma >0$. The corresponding distribution function is $F_x(y)=\Phi\left(\tfrac{y-\beta x}{\sigma}\right)$, where $\Phi$ is the distribution function of the $\mathrm N(0,1)$ law. It is immediately verified that $\mu = \mathrm N\left(0, \tfrac{\sigma^2}{1-\beta^2}\right)$ is a stationary distribution and that $K$ is actually $\mu$-reversible. The kernel can be represented as in \eqref{ker} by the function
		$$
		f(x,U)=F_x^{-1}(U)=\beta x + \sigma \Phi^{-1} (U),
		$$
		where $U$ is a random variable with the distribution $\mathrm U(0,1)$, uniform in $(0,1)$. That is $Y=f(X,U)=\beta X+\sigma\Phi^{-1}(U)$, where $X\sim \mathrm N\left(0, \tfrac{\sigma^2}{1-\beta^2}\right)$ and $U\sim \mathrm U(0,1)$ are independent. By reversibility, when $X \sim \mathrm N\left(0, \tfrac{\sigma^2}{1-\beta^2}\right)$ and the conditional law of $Y$ given $X=x$ is $\mathrm N(\beta x, \sigma^2)$, the conditional law of $X$ given $Y=x$ is the same, thus it has distribution function $\Phi(\tfrac {\cdot -\beta x}{\sigma})$. As a consequence for $(X,U)\sim \mathrm N\left(0, \tfrac{\sigma^2}{1-\beta^2}\right)\otimes\mathrm U(0,1)$ we have
		\begin{align*}
			V=&g_f(x,u)=F_{F_X^{-1}(U)}(X)=\Phi\left(\tfrac{X-\beta F_X^{-1}(U)}{\sigma}\right)\\=& \;\Phi\left(\tfrac{X-\beta (\beta X+\sigma\Phi^{-1}(U)}{\sigma}\right) 
			=\Phi\left(\tfrac{X(1-\beta^2)}{\sigma} -\beta\Phi^{-1}(U)\right).
		\end{align*}
		Note that $\tfrac{X(1-\beta^2)}{\sigma} -\beta\Phi^{-1}(U)\sim\mathrm N(0,1)$. Consequently, $V\sim \mathrm U(0,1)$. Moreover, $Y$ and $V$ are independent since $Y=\beta X+\sigma\Phi^{-1}(U)$ and $\Phi^{-1}(V)=\tfrac{X(1-\beta^2)}{\sigma} -\beta\Phi^{-1}(U)$ are independent. For this, it suffices to note that $$\mathrm{Cov}\left(\beta X+\sigma\Phi^{-1}(U),\;\tfrac{1-\beta^2}{\sigma}\,X -\beta\Phi^{-1}(U)\right)=\tfrac{\beta(1-\beta^2)}{\sigma}\,\mathrm{Var}(X)-\beta\sigma\,\mathrm{Var}(\Phi^{-1}(U))=0.$$  
		Consequently,  the involution $$\mathbf H(x,u)=(\beta x+\sigma\Phi^{-1}(u),\,\Phi\left(\tfrac{1-\beta^2}{\sigma}\,x-\beta\Phi^{-1}(u)\right)),$$ defined on $\mathbb R\times (0,1)$, preserves the product law $\mathrm{N}\left(0,\tfrac{\sigma^2}{1-\beta^2}\right)\otimes \mathrm U(0,1)$.
		\begin{comment}
			By defining 
			$$Z=\sigma \Phi^{-1}(U),\,\,\,W=\sigma \Phi^{-1}(V)=X-\beta Y,$$ we have the alternative way of representing the independence preserving transformation
			$$
			Y=\beta X+ Z,\,\,\, W=(1-\beta^2)X-\beta Z
			$$
			with $Z\sim\mathrm N(0,\sigma^2)$; it is immediately verified that $\mathbf H(x,u)=\left(\beta x+u,\,(1-\beta^2)x-\beta u\right)$ is an involution on $\mathbb R^2$ applied to $(X,Z)\sim \mathrm{N}\left(0,\tfrac{\sigma^2}{1-\beta^2}\right)\otimes\mathrm N(0,\sigma^2)$ and it preserves this product law.
			The last statement can be obtained directly by computing covariances.
		\end{comment}
	\end{example}
	
	\vspace{3mm}\small
	{\bf Acknowledgement:} The research of JW was partially supported by the National Science Center, Poland,  project no. 2023/51/B/ST1/01535. 
	
	\bibliographystyle{amsplain}
	
	\bibliography{RMK_Biblio}
	
\end{document}